\documentclass[12pt, reqno, twoside, letterpaper]{amsart}

\usepackage{paperstyle}

\title[Quadratic nonresidues below the Burgess bound]
      {Quadratic nonresidues below the Burgess bound}
      
\author[W.~D.~Banks]{William D.~Banks}  
\address{Department of Mathematics, University of Missouri, 
Columbia, MO 65211 USA}
\email{bankswd@missouri.edu} 

\author[V.~Z.~Guo]{Victor Z.~Guo} 
\address{Department of Mathematics, University of Missouri, 
Columbia, MO 65211 USA}
\email{zgbmf@mail.missouri.edu} 
        
\date{\today}

\begin{document}

\begin{abstract}
For any odd prime number $p$, let $(\cdot|p)$ be the Legendre symbol,
and let $n_1(p)<n_2(p)<\cdots$ be the sequence of positive nonresidues
modulo $p$, i.e., $(n_k|p)=-1$ for each $k$.
In 1957, Burgess showed that the upper bound $n_1(p)\ll_\eps p^{(4\sqrt{e})^{-1}+\eps}$ holds for
any fixed $\eps>0$. In this paper, we prove that the stronger bound
$$
n_k(p)\ll p^{(4\sqrt{e})^{-1}}\exp\big(\sqrt{e^{-1}\log p\log\log p}\,\big)
$$
holds for all odd primes $p$, where the implied constant is absolute, provided that
$$
k\le p^{(8\sqrt{e})^{-1}}
\exp\big(\tfrac12\sqrt{e^{-1}\log p\log\log p}-\tfrac12\log\log p\big).
$$
For fixed $\eps\in(0,\frac{\pi-2}{9\pi-2}]$ we also show that
there is a number $c=c(\eps)>0$ such that for all odd primes $p$
and either choice of $\theta\in\{\pm 1\}$, there are
$\gg_\eps y/(\log y)^\eps$ natural numbers
$n\le y$ with $(n|p)=\theta$ provided that
$$
y\ge p^{(4\sqrt{e})^{-1}}\exp\big(c(\log p)^{1-\eps}\big).
$$
\end{abstract}

\maketitle


\vskip.5in

\section{Introduction}

For any odd prime $p$, let $n_1(p)$ denote the least positive quadratic
nonresidue modulo $p$; that is,
$$
n_1(p):=\min\{n\in\NN:(n|p)=-1\},
$$
where $(\cdot|p)$ is the \emph{Legendre symbol}.
The first nontrivial bound on $n_1(p)$ was given by
Gauss~\cite[Article~129]{Gauss}, who showed that
$n_1(p)<2\sqrt{p}+1$ holds for every prime $p\equiv 1\pmod 8$.
Vinogradov~\cite{Vino} proved that the bound
$n_1(p)\ll_\eps p^{(2\sqrt{e})^{-1}+\eps}$ holds for any $\eps>0$, and later,
Burgess~\cite{Burg0} extended this range by showing that the bound
\begin{equation}
\label{eq:burg-nr}
n_1(p)\ll_\eps p^{(4\sqrt{e})^{-1}+\eps}
\end{equation}
holds for every fixed $\eps>0$;
this result has not been improved since 1957.

The bound \eqref{eq:burg-nr} implies that the inequality
$n_1(p)\le p^{\frac{1}{4\sqrt{e}}+f(p)}$ holds for all odd primes $p$
with \emph{some} function $f$ such that $f(p)\to 0$ as $p\to\infty$.
Our aim in this note is to improve the bound \eqref{eq:burg-nr}
and to study quadratic nonresidues that lie below
$p^{(4\sqrt{e})^{-1}+\eps}$ for any fixed $\eps>0$.
To this end, let $n_1(p)<n_2(p)<\cdots$ be the sequence of positive nonresidues modulo $p$.

\begin{theorem}
\label{thm:many k}
The bound
\begin{equation}
\label{eq:nkpbnd}
n_k(p)\ll p^{(4\sqrt{e})^{-1}}
\exp\big(\sqrt{e^{-1}\log p\log\log p}\,\big)
\end{equation}
holds for all odd primes $p$ and all positive integers
$$
k\le p^{(8\sqrt{e})^{-1}}
\exp\big(\tfrac12\sqrt{e^{-1}\log p\log\log p}-\tfrac12\log\log p\big),
$$
where the implied constant in \eqref{eq:nkpbnd} is absolute.
\end{theorem}

In a somewhat longer range, we establish the existence of
\emph{many} nonresidues.

\begin{theorem}
\label{thm:y/(log y)^eps}
Let $\eps\in(0,\xi]$ be fixed, where
$$
\xi\defeq\frac{\pi-2}{9\pi-2}=0.04344896\ldots.
$$
There is a number $c=c(\eps)>0$ such that
for all odd primes $p$ and either either choice of $\theta\in\{\pm 1\}$, we have
$$
\#\big\{n\le y:(n|p)=\theta\big\}\gg_\eps\frac{y}{(\log y)^\eps}
\qquad\big(y\ge p^{(4\sqrt{e})^{-1}}\exp\big(c(\log p)^{1-\eps}\big)\big),
$$
where the constant implied by $\gg_\eps$ depends only on $\eps$.
\end{theorem}

\bigskip\noindent\textbf{Acknowledgement.} We thank Kannan Soundararajan,
whose suggestion led to a significant sharpening of our Theorem~\ref{thm:many k}.

\section{Preliminaries}

Throughout the paper, we use the symbols $O$ and $\ll$ with their standard
meanings; any implied constants are \emph{absolute} unless otherwise specified
in the notation.

Throughout the paper, we denote
$$
\lambda\defeq\frac{5\pi-2}{9\pi-2}=0.52172448\ldots,\qquad
\eta\defeq\frac14-\frac1{2\pi}=0.09084505\ldots.
$$
The constant $\eta$ appears in Granville and Soundararajan~\cite[Proposition~1]{GraSou1},
which is one of our principal tools.
In view of the definition of $\xi$ in Theorem~\ref{thm:y/(log y)^eps}, we note that
the following relation holds:
\begin{equation}
\label{eq:relation}
\xi=\eta(1-\lambda)=2\lambda-1.
\end{equation}

In a series of papers, Burgess~\cite{Burg1,Burg2}
established several well known bounds on relatively short character sums
of the form
$$
S_\chi(M,N)\defeq\sum_{M<n\le M+N}\chi(n)\qquad(M,N\in\ZZ,~N\ge 1).
$$
Here, we use a slightly stronger estimate which holds for
characters of prime conductor; see Iwaniec and
Kowalski~\cite[Equation~(12.58)]{IwanKow}.
\begin{lemma}
\label{lem:IwanKow}
Let $\chi$ be a primitive Dirichlet character of prime conductor $p$.
For any integer $r\ge 1$ we have
$$
\big|S_\chi(M,N)\big|\le 30 N^{1-1/r}p^{(r+1)/4r^2}
(\log p)^{1/r}\qquad(M,N\in\ZZ,~N\ge 1).
$$
\end{lemma}

\begin{proposition}
\label{prop:MchiN}
Let $\chi$ be a primitive Dirichlet character of prime conductor $p$,
and put
$$
M_\chi(x)\defeq\frac1x\sum_{n\le x}\chi(n)\qquad(x\ge 1).
$$
Then, uniformly for $c\in[0,(\log p)^{1/3}]$ we have
$$
M_\chi(x)\ll(\log p)^{-c^2}
\qquad\big(x\ge p^{1/4}\exp\big(c\sqrt{\log p\log\log p}\,\big)\big).
$$
\end{proposition}

\begin{proof}
We can assume that $c>0$ else the result is trivial.
Let $z\defeq e^{c\sqrt{\log p\log\log p}}$.  For any integer
$N\ge p^{1/4}z$ we have by Lemma~\ref{lem:IwanKow}:
$$
\big|M_\chi(N)\big|=N^{-1}\big|S_\chi(0,N)\big|\ll
N^{-1/r}p^{(r+1)/4r^2}(\log p)^{1/r}\le
p^{1/4r^2}z^{-1/r}(\log p)^{1/r}
$$
for any integer $r\ge 1$.  We choose
$$
r\defeq\fl{\frac1{2c}\bigg(\frac{\log p}{\log\log p}\bigg)^{1/2}},
$$
where $\fl\cdot$ is the greatest integer function. Since $c\le(\log p)^{1/3}$
it follows that
\begin{align*}
p^{1/4r^2}z^{-1/r}(\log p)^{1/r}
&=\exp\bigg(\frac{\log p}{4r^2}-\frac{c(\log p\log\log p)^{1/2}}{r}
+\frac{\log\log p}{r}\bigg)\\
&=\exp\bigg(\!\!-c^2\log\log p+O\bigg(\frac{(\log\log p)^{3/2}}{(\log p)^{1/6}}\bigg)\bigg)\ll(\log p)^{-c^2}.
\end{align*}
This implies the stated bound.
\end{proof}

\section{Proof of Theorem~\ref{thm:y/(log y)^eps}}

Our proof of Theorem~\ref{thm:y/(log y)^eps}
relies on ideas of Granville and Soundararajan~\cite{GraSou1,GraSou2}.
We begin with a technical lemma.

\begin{lemma}
\label{lem:techlemma}
Let $g$ be a completely multiplicative function
such that $-1\le g(n)\le 1$ for all $n\in\NN$.  Let $x$ be large,
and suppose that $g(p)=1$ for all
$p\le y\defeq \exp\big((\log x)^\lambda\big)$.
Then, uniformly for $1/\sqrt{e}\le\alpha\le 1$, we have
$$
1-\tau(\alpha)+O\big((\log x)^{-\xi}\big)\le M_g(x^\alpha)
\le 1-\tau(\alpha)+\tfrac12\tau(\alpha)^2+O\big((\log x)^{-\xi}\big),
$$
where
$$
\tau(\alpha)\defeq\sum_{p\le x^\alpha}\frac{1-g(p)}{p}.
$$
\end{lemma}

\begin{proof}
Let $\vartheta$ be the Chebyshev function
$\vartheta(u)\defeq\sum_{p\le u}\log p$, and define
$$
\cX(t)\defeq\frac{1}{\vartheta(y^t)}\sum_{p\le y^t}g(p)\log p.
$$
Put $u_\alpha\defeq(\log x^\alpha)/\log y=\alpha(\log x)^{1-\lambda}$.  Using
\cite[Proposition~1]{GraSou1} and taking into account \eqref{eq:relation},
we derive the estimate
\begin{equation}
\label{eq:hydro}
M_g(x^\alpha)=\sigma(u_\alpha)+O\big((\log x)^{-\xi}\big)
\qquad\big(1/\sqrt{e}\le\alpha\le 1\big),
\end{equation}
where $\sigma$ is the unique solution to the integral equation
\begin{align*}
u\sigma(u)=\sigma*\cX(u)=\int_0^u\sigma(u-t)\cX(t)\,dt\qquad&\text{for}\quad u>1,\\
\text{with the initial condition}\quad\sigma(u)=1\quad&\text{for}\quad 0\le u\le 1.
\end{align*}
Moreover, using \cite[Proposition~3.6]{GraSou1} we see that
$$
1-I_1(u_\alpha;\cX)\le\sigma(u_\alpha)\le 1-I_1(u_\alpha;\cX)+I_2(u_\alpha;\cX),
$$
where
\begin{align*}
I_1(u;\cX)&\defeq\int_1^u\frac{1-\cX(t)}{t}\,dt,\\
I_2(u;\cX)&\defeq\mathop{\int_1^u\int_1^u}\limits_{(t_1+t_2\le u)}
\frac{1-\cX(t_1)}{t_1}\frac{1-\cX(t_2)}{t_2}\,dt_1\,dt_2.
\end{align*}
Removing the condition $t_1+t_2\le u$ we derive that
$I_2(u;\cX)\le I_1(u;\cX)^2$; hence, in view of the trivial
bound $\tau(\alpha)\ll\log\log x$ it suffices to establish the uniform estimate
\begin{equation}
\label{eq:I1est}
I_1(u_\alpha;\cX)=\tau(\alpha)+O\big((\log y)^{-1}\big)
\qquad\big(1/\sqrt{e}\le\alpha\le 1\big).
\end{equation}
For this, put $S(v)\defeq\sum_{p\le v}(1-g(p))\log p$, and note that
\begin{align*}
\tau(\alpha)&=\int_y^{x^\alpha}\frac{dS(v)}{v\log v}
=\bigg[\frac{S(v)}{v\log v}\bigg]_y^{x^\alpha}+
\int_y^{x^\alpha}\frac{S(v)(\log v+1)}{(v\log v)^2}\,dv\\
&=\int_y^{x^\alpha}\frac{S(v)}{v^2\log v}\,dv
+O\big((\log y)^{-1}\big),
\end{align*}
where we have used the bound $S(v)\ll v$. Making the change of variables
$v=y^t$, $dv=y^t\log y\,dt$, and taking into account that
$S(y^t)=\vartheta(y^t)(1-\cX(t))$, we have
\begin{align*}
\tau(\alpha)
&=\int_1^{u_\alpha}\frac{\vartheta(y^t)}{y^t}\frac{1-\cX(t)}{t}\,dt
+O\big((\log y)^{-1}\big).
\end{align*}
The estimate \eqref{eq:I1est} now follows from the Prime Number Theorem
in the form $\vartheta(y^t)=y^t+O(y^t/\log y^t)$.
\end{proof}

The next statement is a variant of \cite[Proposition~7.1]{GraSou2}.

\begin{proposition}
\label{prop:GraSou}
Let $x$ be large, and let $f$ be a completely multiplicative function
such that $-1\le f(n)\le 1$ for all $n\in\NN$.
Then, uniformly for
$1/\sqrt{e}\le\alpha\le 1$, we have
\begin{equation}
\label{eq:Mfxa-est}
\big|M_f(x^\alpha)\big|\le\max\big\{|\delta_1|,\tfrac12+2(\log\alpha)^2\big\}+
O\Big(\max\big\{\big|M_f(x)\big|,(\log x)^{-\xi}\big\}\Big),
\end{equation}
where
$$
\delta_1\defeq1-2\log(1+\sqrt{e})+4\int_1^{\sqrt{e}}\frac{\log u}{u+1}\,du=
-0.656999\ldots.
$$
\end{proposition}

\begin{proof}
We follow the proof of \cite[Proposition~7.1]{GraSou2} closely, making use of
the work in \cite{GraSou1}.  Let
$y\defeq \exp\big((\log x)^\lambda\big)$, and let $g$ be
the completely multiplicative function defined by
$$
g(p)\defeq\begin{cases}
1&\hbox{if $p\le y$},\\
f(p)&\hbox{if $p>y$}.
\end{cases}
$$
Using \cite[Proposition~4.4]{GraSou1} (with $S=[-1,1]$ and $\varphi=\pi/2$)
 and taking into account \eqref{eq:relation},
we derive the estimate
$$
M_f(x^\alpha)=\Theta(f,y)M_g(x^\alpha)+O\big((\log x)^{-\xi}\big),
$$
where
$$
\Theta(f,y)\defeq\prod_{p\le y}\bigg(1-\frac1p\bigg)\bigg(1-\frac{f(p)}{p}\bigg)^{-1}.
$$
Since $\big|M_g(x^\alpha)\big|\le 1$,
we obtain \eqref{eq:Mfxa-est} whenever $\Theta(f,y)\le\frac12$;
thus, we can assume without loss of generality that
$\Theta(f,y)\in[\frac12,1]$, and it suffices to show that
\begin{equation}
\label{eq:Mgxa-est}
\big|M_g(x^\alpha)\big|\le\max\big\{|\delta_1|,\tfrac12+2(\log\alpha)^2\big\}+O(B)\end{equation}
holds uniformly for $1/\sqrt{e}\le\alpha\le 1$,
where
$$
B\defeq\max\big\{\big|M_f(x)\big|,(\log x)^{-\xi}\big\}.
$$

Applying Lemma~\ref{lem:techlemma} with $\alpha=1$, we have
$$
\tau(1)\ge 1+O(B).
$$
Further, by Mertens' theorem we have for $1/\sqrt{e}\le\alpha\le 1$:
$$
\tau(1)-\tau(\alpha)=\sum_{x^\alpha<p\le x}\frac{1-g(p)}{p}
\le\sum_{x^\alpha<p\le x}\frac{2}{p}=2\log\alpha+O\big((\log x)^{-1}\big).
$$
Consequently,
$$
\tau(\alpha)\ge 1-2\log\alpha+O(B).
$$

Using \cite[Theorem~5.1]{GraSou1} together with \eqref{eq:hydro},
if $\tau(\alpha)\ge 1$ we have
$$
\big|M_g(x^\alpha)\big|\le|\delta_1|+O\big((\log x)^{-\xi}\big)
=|\delta_1|+O(B).
$$
On the other hand, if $1-2\log\alpha+O(B)\le\tau(\alpha)\le 1$
we can apply Lemma~\ref{lem:techlemma} again to conclude that
\begin{align*}
\big|M_g(x^\alpha)\big|
&\le 1-\tau(\alpha)+\tfrac12\tau(\alpha)^2+O(B)
\le\tfrac12+2(\log\alpha)^2+O(B).
\end{align*}
Putting these estimates together, we obtain \eqref{eq:Mgxa-est},
which finishes the proof.
\end{proof}

\begin{proof}[Proof of Theorem~\ref{thm:y/(log y)^eps}]
Let $\chi\defeq(\cdot|p)$, and let $\eps\in(0,\xi]$ and $\theta\in\{\pm 1\}$
be fixed. Since
\begin{equation}
\label{eq:fork}
\#\big\{n\le y:\chi(n)=\theta\big\}
=\tfrac12y\big(1+\theta M_\chi(y)+O(p^{-1})\big),
\end{equation}
the result is easily proved for $y>p$ (e.g., using
Proposition~\ref{prop:MchiN}).  Thus, we can assume
$y\le p$ in what follows. Moreover, it suffices to prove
the theorem for all sufficiently large primes $p$
(depending on $\eps$).

Let $\alpha\in[\tfrac{1}{\sqrt{e}},1]$ and put $x\defeq y^{1/\alpha}$.
Note that $\log p\asymp\log x\asymp\log y$ since $p^{(4\sqrt{e})^{-1}}\le y\le p$.
Applying Proposition~\ref{prop:MchiN},
the bound $M_\chi(x)\ll (\log x)^{-\eps}$ holds provided that
\begin{equation}
\label{eq:check}
x\ge p^{1/4}\exp\big(\eps^{-1/2}\sqrt{\log p\log\log p}\big),
\end{equation}
which we assume for the moment. Since $\eps\le\xi$,
Proposition~\ref{prop:GraSou} yields the bound
$$
\big|M_\chi(y)\big|=\big|M_\chi(x^\alpha)\big|
\le\max\big\{|\delta_1|,\tfrac12+2(\log\alpha)^2\big\}+c_1(\log p)^{-\eps}
$$
with some number $c_1>0$ that depends only on $\eps$.
Taking $\alpha\defeq \tfrac{1}{\sqrt{e}}+c_1(\log p)^{-\eps}$,
for all sufficiently large $p$ (depending on $\eps$) we have
$$
\big|M_\chi(y)\big|
\le 1-(2\sqrt{e}-1)c_1(\log p)^{-\eps}+
O_\eps\big((\log x)^{-2\eps}\big).
$$
In particular, for some sufficiently large $c_2>0$ (depending on $\eps$) the bound
$$
\big|M_\chi(y)\big|=\big|M_\chi(x^\alpha)\big|\le 1-c_2(\log y)^{-\eps}
$$
holds.  In view of \eqref{eq:fork}
we obtain the stated result.

To verify \eqref{eq:check}, observe that
$\alpha^{-1}\ge\sqrt{e}-c_3(\log p)^{-\eps}$ with some number $c_3>0$
that depends only on $\eps$. If $c>0$ 
and $y\ge p^{(4\sqrt{e})^{-1}}e^{c(\log p)^{1-\eps}}$, then
\begin{align*}
\log x=\alpha^{-1}\log y&\ge 
\big(\tfrac1{4\sqrt{e}}\log p+c(\log p)^{1-\eps}\big)
\big(\sqrt{e}-c_3(\log p)^{-\eps}\big)\\
&=\tfrac14\log p+\big(c\sqrt{e}-\tfrac{c_3}{4\sqrt{e}}
-cc_3(\log p)^{-\eps}\big)(\log p)^{1-\eps}.
\end{align*}
Hence, if $c$ and $p$ are large enough, depending only on $\eps$, then
$$
\log x\ge \tfrac14\log p+\eps^{-1/2}\sqrt{\log p\log\log p}
$$
as required.
\end{proof}

\section{Proof of Theorem~\ref{thm:many k}}

Let $C>0$ be a fixed (absolute) constant to be determined below.
Put
$$
E\defeq p^{(4\sqrt{e})^{-1}}
\exp\big(\sqrt{e^{-1}\log p\log\log p}\,\big)\mand
B\defeq E^{1/2}(\log p)^{1/2}.
$$
Let $N\defeq n_1(p)$ and $M\defeq n_k(p)$, where $k$ is a 
positive integer such that
\begin{equation}
\label{eq:kcond}
k\le CE^{1/2}(\log p)^{-1/2}.
\end{equation}
To prove the theorem we need to show that $M\ll E$.

\emph{Case 1: $N\le B$}.
If the interval $[1,2k]$ contains at least $k$ nonresidues, then
$$
M\le 2k\ll E^{1/2}(\log p)^{-1/2}\ll E
$$
and we are done.  If the interval $[1,2k]$ contains fewer than
$k$ nonresidues, then $[1,2k]$ contains at least $k$ residues $m_1,\ldots,m_k$.  Therefore,
$Nm_1,\ldots,Nm_k$ are all nonresidues in $[1,2kB]$,
and we have (using \eqref{eq:kcond} and the definition of $B$)
$$
M\le 2kB\ll E.
$$

\emph{Case 2: $N>B$}.
Applying Theorem~\ref{thm:y/(log y)^eps} with $\eps\defeq\xi$,
$y\defeq B^{5/2}$ and $\theta\defeq-1$, 
there is an absolute constant $c_1>0$ such that
$$
\#\big\{n\le B^{5/2}:(n|p)=-1\big\}\gg\frac{B^{5/2}}{(\log B)^\xi}
$$
provided that
$$
B^{5/2}\ge p^{(4\sqrt{e})^{-1}}\exp\big(c_1(\log p)^{1-\xi}\big).
$$
Since $B^{5/2}>E^{5/4}$ the latter inequality is easily satisfied
for all large $p$; thus, if $p$ is large enough, then
the $k$-th nonresidue $M=n_k(p)$ satisfies
$$
N\le M\le B^{5/2}<N^{5/2}.
$$

Let $x\in(M,N^3)$, and note that $\log x\asymp\log p\asymp\log N$.
Following an idea of Vinogradov, we see that the inequality $x<N^3$
implies that every nonresidue $n\le x$ can be uniquely represented in the
form $n=qm$, where $q$ is a prime nonresidue, and $m$ is a positive integer residue
not exceeding $x/q$; this leads to the lower bound
$$
\sum_{n\le x}(n|p)
\ge x-2\sum_{\substack{N\le q\le x\\(q|p)=-1}}\frac{x}{q}+O(1).
$$
Since $M=n_k(p)$, there are at most $k$ prime nonresidues in $[N,M]$; thus,
$$
\sum_{n\le x}(n|p)
\ge x-\frac{2kx}{N}
-2\sum_{M<q\le x}\frac{x}{q}+O(1).
$$
Recalling that $N>B=E^{1/2}(\log p)^{1/2}$
and using \eqref{eq:kcond} together
with the Prime Number Theorem, we derive the lower bound
$$
\sum_{n\le x}(n|p)\ge 
x \bigg(1 -\frac{2C}{\log p}
- 2 \log \frac{\log x}{\log M}\bigg)
+ O\bigg(\frac{x}{(\log x)^{100}}\bigg).
$$
Now let $x\defeq e^{-3C}M^{\sqrt{e}}$.
Since $-2\log(1-t)\ge 2t$ for all $t\in[0,\tfrac12]$, for any
sufficiently large $p$ (depending on the choice of $C$) we have
$$
1-2\log \frac{\log x}{\log M}
=-2\log\bigg(1-\frac{3C}{\sqrt{e}\log M}\bigg)
\ge\frac{6C}{\sqrt{e}\log M},
$$
and thus
$$
\frac1x\sum_{n\le x}(n|p)\ge\frac{6C}{\sqrt{e}\log M}-\frac{2C}{\log p}
+ O\bigg(\frac{1}{(\log x)^{100}}\bigg).
$$
Since $M\le B^{5/2}\le p$ for all large $p$, it follows that
$$
\frac1x\sum_{n\le x}(n|p)\ge\frac{C}{\log p}
$$
if $p$ is large enough (depending on $C$). On the other hand,
using Proposition~\ref{prop:MchiN} with $c=1$, we see that
there is an absolute constant $C_0>0$ such that
$$
\frac1x\sum_{n\le x}(n|p)\le\frac{C_0}{\log p}
$$
whenever $x\ge p^{1/4}e^{\sqrt{\log p\log\log p}}$.  If $C$ is initially
chosen so that $C>C_0$, then these two bounds are incompatible unless
$$
e^{-3C}M^{\sqrt{e}}<p^{1/4}\exp\big(\sqrt{\log p\log\log p}\,\big).
$$
The theorem follows.

\end{document}